\documentclass{amsart}
\usepackage[english]{babel}
\usepackage{graphicx}
\usepackage{amsmath}
\usepackage{amsthm}
\usepackage{amssymb}
\newtheorem{Lemma 1}{Lemma}[section]
\newtheorem{Lemma 2}[Lemma 1]{Lemma}
\newtheorem{Lemma 3}[Lemma 1]{Lemma}
\newtheorem{Lemma 4}[Lemma 1]{Lemma}
\newtheorem{Main}[Lemma 1]{Theorem}
\newtheorem{Lemma 5}[Lemma 1]{Lemma}
\begin{document}
\title{ On the distribution of numbers related to the divisors of $x^n-1$ }
\address{Department of Mathematics, Indian Institute of Technology Roorkee,India 247667}
\author{Sai Teja Somu}
\date{November 9, 2015}
\maketitle
\begin{abstract}
Let $n_1,\cdots,n_r$ be any finite sequence of integers and
let $S$ be the set of all natural numbers $n$ for which there exists a divisor $d(x)=1+\sum_{i=1}^{deg(d)}c_ix^i$ of $x^n-1$ such that $c_i=n_i$ for $1\leq i \leq r$. In this paper we show that the set $S$ has a natural density. Furthermore, we find the value of the natural density of $S$. 
\end{abstract}
\section{Introduction}
Cyclotomic polynomials arise naturally as irreducible divisors of $x^n-1$. The polynomial $x^n-1$ can be factored in the following way
\begin{equation}
x^n-1=\prod_{d|n}\phi_d(x).
\end{equation}
Applying Mobius inversion we get
\begin{equation}
\phi_n(x)=\prod_{d|n}(x^d-1)^{\mu(\frac{n}{d})}.
\end{equation}
 The problem of determining size of maximum coefficient of cyclotomic polynomials has been the subject of the papers \cite{B} and \cite{E}. In \cite{D} Pomerance and Ryan study the size of maximum coefficient of divisors of $x^n-1$. 

 It has been proven that for every finite sequence of integers $(n_i)_{i=1}^{r}$ , there exists $d(x)=1+\sum\limits_{i=1}^{deg(d)}c_ix^i$, a divisor of $x^n-1$ for some $n\in \mathbb{N}$, such that $c_i=n_i$ for  $1\leq i \leq r$. In this paper we investigate the following problem. For a given sequence $(n_i)_{i=1}^{r}$, let
$S(n_1,\cdots,n_r)$ denote the set of all $n$ such that $x^n-1$ has a divisor $d(x)$ of the form $d(x)=1+\sum\limits_{i=1}^{r}n_ix^i+\sum\limits_{i=r+1}^{deg(d)}c_ix^i$. We prove that $S(n_1,\cdots,n_r)$ has a natural density. Observe that if $n\in S(n_1,\cdots,n_r)$ then every multiple of $n$ is in $S(n_1,\cdots,n_r)$. 

\section{Notation}

If $f(x)$ and $g(x)$ are two analytic functions in some neighborhood of $0$, we denote $f(x)\equiv g(x) \mod x^{r+1}$ if the coefficients of $x^i$ in the power series of $f(x)$ and $g(x)$ are equal for $0\leq i \leq r$. 

We denote $\omega(n)$ for number of distinct prime factors of $n$. Let $\delta(d)$ be $1$ if $d\neq 1$ and $\delta(d)$ be $-1$ otherwise. Note that
\begin{equation}
\phi_n(x)=\delta(n)\prod_{d|n}\left(1-x^d\right)^{\mu\left(\frac{n}{d}\right)}.
\end{equation} 
\section{Proof of Main Theorem}
We require several lemmas in order to prove that $S(n_1,\cdots,n_r)$ has a natural density.
\begin{Lemma 1}\label{Lemma 1}
For every finite sequence of integers $n_1,\cdots,n_r$ there exists a unique sequence of integers $k_1,\cdots,k_r$ such that
\begin{equation}\label{equation 1}
\prod\limits_{i=1}^{r}(1-x^i)^{k_i}\equiv 1+\sum_{i=1}^{r}n_ix^i \mod x^{r+1}.
\end{equation}
\end{Lemma 1}
\begin{proof}
The proof that there exists a sequence $k_1,\cdots,k_r$ is by induction on $r$. If $r=1$ and $n_1\in \mathbb{Z}$ then $(1-x)^{-n_1}\equiv 1+n_1x \mod x^2$ hence the existence part of lemma is true for $r=1$. If we assume that the existence part of lemma is true for $r$, then for any sequence of $r+1$ integers $(n_i)_{i=1}^{r+1}$,  there exist $r$ integers $k_1,\cdots,k_r$ such that 
\[ \prod\limits_{i=1}^{r}(1-x^i)^{k_i}\equiv 1+\sum_{i=1}^{r}n_ix^i \mod x^{r+1}.\]
Let $n'_{r+1}$ be an integer such that 
\[ \prod\limits_{i=1}^{r}(1-x^i)^{k_i}\equiv 1+\sum_{i=1}^{r}n_ix^i+n'_{r+1}x^{r+1} \mod x^{r+2}.\]
We have 
\[\prod\limits_{i=1}^{r}(1-x^i)^{k_i}(1-x^{r+1})^{n'_{r+1}-n_{r+1}}\equiv 1+\sum_{i=1}^{r+1}n_ix^i\mod x^{r+2}.\]
Hence the existence part of the lemma is true for $r+1$.

For the uniqueness part,
if there are two finite sequences $k_1,\cdots,k_r$ and $k'_1,\cdots,k'_r$ such that 
\[ \prod\limits_{i=1}^{r}(1-x^i)^{k_i}\equiv \prod\limits_{i=1}^{r}(1-x^i)^{k'_i} \mod x^{r+1}.\]

If the two sequences are distinct then let $i$ be the least index such that $k_i-k'_i\neq 0$ then  we have 
 \begin{align*}
 \prod_{j=i}^{r}(1-x^i)^{k_i-k'_i}\equiv 1 \mod x^{i+1}
 \end{align*}  
or 
\[1-(k_i-k'_i)x^{i}\equiv 1 \mod x^{i+1}\] which implies $k_i-k'_i=0$ contradicting the assumption that $k_i-k'_i\neq0$.
\end{proof}

For a given sequence $n_1,\cdots,n_r$ we proved that there exists a unique sequence $k_1(n_1,\cdots,n_r),\cdots,k_r(n_1,\cdots,n_r)$ such that equation (\ref{equation 1}) is true. Let $A(n_1,\cdots,n_r)$ be the set defined by $A(n_1,\cdots,n_r):=\{ 1\leq i\leq r : k_i(n_1,\cdots,n_r)\neq 0 \}$. If the set $A(n_1,\cdots,n_r)$ is non empty let $l(n_1,\cdots,n_r)$ be the least common multiple of elements of $A(n_1,\cdots,n_r)$, otherwise let $l(n_1,\cdots,n_r)$ be $1$.

\begin{Lemma 2}\label{Lemma 2}
If $n\in S(n_1,\cdots,n_r)$ then $n$ is a multiple $l(n_1,\cdots,n_r)$.
\end{Lemma 2}
\begin{proof}
We will prove that if $l(n_1,\cdots,n_r)\nmid n$ then $n\notin S(n_1,\cdots,n_r)$.
 If $l(n_1,\cdots,n_r)$ does not divide $n$ then there exists an $i\in A(n_1,\cdots,n_r)$ such that $i\nmid n$. That is, $k_i(n_1,\cdots,n_r)\neq 0$ and $i\nmid n$. 

Any divisor $d(x)$ of $x^n-1$ such that $d(0)=1$ will be of the form \[d(x)=\prod_{d\in S}\delta(d)\phi_d(x),\] where $S$ is some subset of set of divisors of $n$.
Hence \begin{align*}
d(x)&=\prod_{d\in S}\delta(d)\phi_d(x)\\
&=\prod_{d\in S}\prod_{d'|d}\left(1-x^{d'}\right)^{\mu(\frac{d}{d'})}\\
&\equiv \prod_{1\leq d' \leq r}\prod_{\substack{d\equiv 0 \mod d'\\d\in S}}\left(1-x^{d'}\right)^{\mu(\frac{d}{d'})} \mod x^{r+1}\\
&\equiv \prod_{j=1}^{r}(1-x^j)^{l_j}\mod x^{r+1},
\end{align*}
where $l_m=\sum_{\substack{d\in S\\d\equiv 0 \mod m}}\mu\left(\frac{d}{m}\right)$ for $1\leq m\leq r$.
Therefore as $i\nmid n$, $l_i=0.$ Hence $l_i\neq k_i(n_1,\cdots,n_r)$ and from uniqueness part of Lemma \ref{Lemma 1} we have $d(x)\not \equiv 1+\sum_{j=1}^{r}n_jx^j\mod x^{r+1}$. Hence $n\notin S(n_1,\cdots,n_r)$.
\end{proof} 
\begin{Lemma 3}\label{Lemma 3}
If $p_1,\cdots,p_s$ are distinct primes greater than $r$ not dividing $d$ and $q_1,\cdots,q_s$ are distinct primes  greater than $r$ and not dividing $d$ then for all natural numbers  $e_1,\cdots,e_s$ we have
 $\phi_{dp_1^{e_1}\cdots p_s^{e_s}}(x)\equiv \phi_{dq_1^{e_1}\cdots q_s^{e_s}}(x) \mod x^{r+1}.$ 
\end{Lemma 3}
\begin{proof}
For every divisor $d'$ of $d$ we have $\mu\left(\frac{dp_1^{e_1}\cdots p_s^{e_s}}{d'}\right)=\mu\left(\frac{dq_1^{e_1}\cdots q_s^{e_s}}{d'}\right).$  From equation (2)
\begin{align*}
\phi_{dp_1^{e_1}\cdots p_s^{e_s}}(x)&\equiv\prod_{d'|d}\left(1-x^{d'}\right)^{\mu\left(\frac{dp_1^{e_1}\cdots p_s^{e_s}}{d'}\right)}\mod x^{r+1}\\
&\equiv \prod_{d'|d}\left(1-x^{d'}\right)^{\mu\left(\frac{dq_1^{e_1}\cdots q_s^{e_s}}{d'}\right)}\mod x^{r+1}\\
& \equiv \phi_{dq_1^{e_1}\cdots q_s^{e_s}}(x) \mod x^{r+1}.
\end{align*}
\end{proof}

\begin{Lemma 5}\label{Lemma 5}
  If $p_1$ and $p_2$ are two distinct primes greater than $r$ and if $d\leq r$ then $\phi_{dp_1p_2}(x)\equiv \delta(d)\phi_d \mod x^{r+1}.$ 
\end{Lemma 5}
\begin{proof}
From (2) we have
\begin{align*}
\phi_{dp_1p_2}(x)&\equiv \prod_{d'|d}\left(1-x^{d'}\right)^{\mu\left(\frac{dp_1p_2}{d'}\right)}\mod x^{r+1}\\
&\equiv \prod_{d'|d}\left(1-x^{d'}\right)^{\mu\left(\frac{d}{d'}\right)}\mod x^{r+1}\\
&\equiv \delta(d) \phi_d(x) \mod x^{r+1}.
\end{align*}
\end{proof}

\begin{Lemma 4}\label{Lemma 4}
For every finite sequence $n_1,\cdots,n_r$ there exist $k$ distinct primes $q_1,\cdots,q_k$ greater than $r$ such that $n=l(n_1,\cdots,n_r)q_1q_2\cdots q_k\in S(n_1,\cdots,n_r)$.
\end{Lemma 4}
\begin{proof}
From Lemma \ref{Lemma 1} we have 
$ \prod\limits_{i=1}^{r}(1-x^i)^{k_i}\equiv 1+\sum_{i=1}^{r}n_ix^i \mod x^{r+1},$ where $k_i=k_i(n_1,\cdots,n_r)$.
From the definition of $A(n_1,\cdots,n_r)$, $k_i\neq 0$ if and only if $i\in A(n_1,\cdots,n_r)$. Let $i_1,\cdots,i_p$ be the elements of $A(n_1,\cdots,n_r)$. We have
    \begin{equation}\label{equation 2}
    1+\sum_{i=1}^{r}n_ix^i \equiv \prod\limits_{j=1}^{p}(1-x^{i_j})^{k_{i_j}} \mod x^{r+1} .
    \end{equation}
    
Let $r^{(j)}_1,\cdots,r^{(j)}_{|k_j|}$ for $ 1\leq j\leq p$ be numbers such that for $1\leq a\leq |k_{i_{j_1}}|$ and $1\leq b\leq |k_{i_{j_2}}|$,     $r^{(j_1)}_{a}=r^{(j_2)}_{b}$ if and only if $j_1=j_2$ and $a=b$. If $k_{i_{j}}>0$ then $r^{(j)}_{a}$ is a product of two distinct primes and each prime factor of $r^{(j)}_{a}$ is greater than $r$. If $k_{i_{j}}<0$ then $r^{(j)}_a$ is a prime number greater than $r$.

If $k_{i_j}>0$ then let 
\begin{equation}\label{equation 3}
d_j(x)=\prod_{m=1}^{k_{i_j}}\prod_{d|i_j}\phi_{dr^{(j)}_m}(x) .
\end{equation}
If $k_{i_j}>0$ then as $r^{(j)}_m$ is a product two prime factors greater from Lemma \ref{Lemma 5} we have
$\phi_{dr^{(j)}_m}(x)\equiv \delta(d)\phi_{d}(x) \mod x^{r+1} .$
Therefore 
\begin{align*}
\prod_{d|i_j}\phi_{dr^{(j)}_m}(x)&\equiv \prod_{d|i_j}\delta(d)\phi_{d}(x) \mod x^{r+1}\\
&\equiv (1-x^{i_j})\mod x^{r+1}.
\end{align*}
Hence from (\ref{equation 3}) we have
\begin{equation}\label{equation 4}
d_j(x)\equiv \prod_{m=1}^{k_{i_j}}(1-x^{i_j})\equiv (1-x^{i_j})^{k_{i_j}} \mod x^{r+1}.
\end{equation}
If $k_{i_j}<0$ let 
\[d_j(x)=\prod\limits_{m=1}^{-k_{i_j}}\prod_{d|i_j}\phi_{dr^{(j)}_m}(x) \mod x^{r+1}.\]
As $k_{i_j}<0$, $r^{(j)}_m$ is a prime number greater than $r$. Hence
\begin{align*}
\prod_{d|i_j}\phi_{dr^{(j)}_m}(x)&=\frac{\prod_{d|i_jr^{(j)}_m}\phi_{d}(x)}{\prod_{d|i_j}\phi_{d}(x)}\\
&\equiv \frac{\left(x^{i_jr^{(j)}_m}-1\right)}{(x^{i_j}-1)}\mod x^{r+1}\\
&\equiv \left(1-x^{i_j}\right)^{-1} \mod x^{r+1}.
\end{align*}
Therefore 

\begin{equation}\label{equation 5}
d_j(x)\equiv\prod\limits_{m=1}^{-k_{i_j}}\left(1-x^{i_j}\right)^{-1}\equiv \left(1-x^{i_j}\right)^{k_{i_j}} \mod x^{r+1}.
\end{equation}
From (\ref{equation 2}), (\ref{equation 4}) and (\ref{equation 5}) we have
\begin{equation}\label{equation 6}
d(x)=\prod_{j=1}^{p}d_j(x)\equiv \prod_{j=1}^{p}(1-x^{i_j})^{k_{i_j}}\equiv 1+\sum\limits_{i=1}^{r}n_ix^i \mod x^{r+1}.
\end{equation}
If the set $\{i_jr^{(j)}_m: 1\leq j \leq p, 1\leq m\leq |k_{i_j}| \}$ is non empty, let $n$ be the least common multiple of the elements of the set and let $n=1$ if the set is empty. Clearly $d(x)$ is a divisor of $x^n-1$ and therefore $n\in S(n_1,\cdots,n_r)$. Observe that $n$ is of the form $l(n_1,\cdots,n_r)q_1q_2\cdots q_k$ where $q_i$'s are distinct prime factors greater than $r$. 
\end{proof}

\begin{Main}\label{Main}
For every finite sequence $n_1,\cdots,n_r$, let
$N( n_1,\cdots,n_r,x)$ denote number of $n\leq x$ such that $n\in S(n_1,\cdots,n_r)$. There exists a $k\in \mathbb{N}$ such that \[N(n_1,\cdots,n_r,x)=C(n_1,\cdots,n_r)x+O\left(\frac{x(\log\log x)^{k}}{\log x}\right), \] where $C(n_1,\cdots,n_r)=\frac{1}{l(n_1,\cdots,n_r)}$.  
\end{Main}
\begin{proof}
For brevity, let $S(n_1,\cdots,n_r)=S$ and $l(n_1,\cdots,n_r)=l$.
From Lemma \ref{Lemma 4} there exists an $m_1$ of the form $m_1=lq_1\cdots q_k$ and a divisor $d_1(x)$ of $x^{m_1}-1$ such that
\begin{equation}\label{equation 7} 
d_1(x)\equiv 1+\sum_{i=1}^{p}n_ix^i \mod x^{r+1}.
\end{equation}
 For every $m_2$ of the form $m_2=lp_1\cdots p_k$ such that $p_1,\cdots,p_k$ are distinct primes greater than $r$. Let $S_1$ be the set of divisors of $m_1$ and $S_2$ be the set of divisors of $m_2$. Let $g:S_1\rightarrow S_2$ be a map defined as follows. As $l$ and $q_1\cdots q_k$ are relatively prime, every divisor of $d$ of $lq_1\cdots q_k$ can be uniquely written in the form $d=d'_1q_{i_1}\cdots q_{i_s}$ where $d'_1$ divides $l$. Define $g(d'_1q_{i_1}\cdots q_{i_s})=d'_1p_{i_1}\cdots p_{i_s}$.
 From Lemma \ref{Lemma 3} it follows that 
 $\phi_{d}(x)\equiv \phi_{g(d)}(x)\mod x^{r+1}$.
 As $d_1(x)$ is of the form $\prod_{d'\in R_1}\delta(d')\phi_{d'}(x)$ where $R_1$ is a subset of $S_1$ there will be $d_2(x)=\prod_{d'\in R_1}\delta(g(d'))\phi_{g(d')}(x)$, a divisor of $x^{m_2}-1$, and $d_2(x)\equiv d_1(x)\equiv 1+\sum_{i=1}^{r}n_ix^i \mod x^{r+1}$. Therefore every number of the form $lp_1\cdots p_k$ where $p_i$'s are distinct primes greater than $r$ belongs to $S$ which implies that every number $lm$ belongs to $S$, if number of distinct prime factors of $m$ greater than $r$ is at least $k$. Hence if $\omega(m)\geq r+k$ then $lm\in S$ as $\omega(m)\geq r+k$ implies that number of prime factors of $m$ greater than $r$ is at least $k$. From 3.1. Lemma B of \cite{C} \[N(n_1,\cdots,n_r,x)\geq\left|\{lm\leq x: \omega(m)\geq r+k\}\right|=\frac{x}{l}+O\left(\frac{x(\log\log x)^{r+k-1}}{\log x}\right).\]
 From Lemma \ref{Lemma 2}, if $n\in S$ then $l|n$ which implies that $N(n_1,\cdots,n_r,x)\leq \frac{x}{l}$. Combining the two inequalities we get
 \[N(n_1,\cdots,n_r,x)=\frac{x}{l}+O\left(\frac{x(\log\log x)^{r+k-1}}{\log x}\right)\] which completes the proof of the theorem.
\end{proof}

\end{document}